\documentclass[11pt,reqno]{amsart}

\usepackage{amssymb}
\usepackage{amscd}
\usepackage{amsfonts}
\usepackage{version}

\numberwithin{equation}{section} \theoremstyle{plain}
\newtheorem{theorem}[subsection]{Theorem}
\newtheorem{proposition}[subsection]{Proposition}
\newtheorem{lemma}[subsection]{Lemma}

\newtheorem{definition}[subsection]{Definition}

\newtheorem*{mainthm3-repeat}{Theorem \ref{mainthm3}}
\newtheorem*{gromov-cor-repeat}{Corollary \ref{gromov-cor}}
\newtheorem*{sum-product-rpt}{Theorem \ref{sum-product-fp}}
\newtheorem*{mainthm-repeat}{Theorem \ref{mainthm2}}
\newtheorem*{virtually-solvable-repeat}{Lemma \ref{virtually-solvable}}

\renewcommand{\leq}{\leqslant}
\renewcommand{\geq}{\geqslant}
\newsavebox{\proofbox}
\savebox{\proofbox}{\begin{picture}(7,7)  \put(0,0){\framebox(7,7){}}\end{picture}}

\newcommand\C{\mathbb{C}}
\newcommand\N{\mathbb{N}}

\newcommand\SL{\operatorname{SL}}

\newcommand\GL{\operatorname{GL}}

\newcommand\eps{\varepsilon}

\def\endproof{\hfill{\usebox{\proofbox}}\vspace{11pt}}

\begin{document}

\title[A note on approximate subgroups]{A note on approximate subgroups of $\GL_n(\C)$ and uniformly nonamenable groups}
\author{Emmanuel Breuillard}
\address{Laboratoire de Math\'ematiques\\
B\^atiment 425, Universit\'e Paris Sud 11\\
91405 Orsay\\
FRANCE}
\email{emmanuel.breuillard@math.u-psud.fr}

\author{Ben Green}
\address{Centre for Mathematical Sciences\\
Wilberforce Road\\
Cambridge CB3 0WA\\
England }
\email{b.j.green@dpmms.cam.ac.uk}

\author{Terence Tao}
\address{Department of Mathematics, UCLA\\
405 Hilgard Ave\\
Los Angeles CA 90095\\
USA}
\email{tao@math.ucla.edu}

\begin{abstract}  The aim of this brief note is to offer another proof of a theorem of Hrushovski that approximate subgroups of $\GL_n(\C)$ are almost nilpotent. This approach generalizes to uniformly non amenable groups.
\end{abstract}

\maketitle
\tableofcontents

\section{Introduction}

Throughout this paper $K \geq 2$ is a real number. We begin by recalling, very briefly, the notions of $K$-\emph{approximate group} and of \emph{control}. For a more leisurely introduction we refer the reader to our paper \cite{bgt-long}.

\begin{definition}
Suppose that $A$ is a finite set in some group which is symmetric in the sense that the identity lies in $A$ and $a^{-1} \in A$ whenever $a \in A$. Then we say that $A$ is a $K$-approximate group if there is some  symmetric set $X$, $|X| \leq K$, such that $A^2 \subseteq AX$. If $B$ is some other set then we say that $A$ is $K$-controlled by $B$ if $|B| \leq K|A|$ and if there is some set $X$, $|X| \leq K$, such that $A \subseteq BX \cap XB$.
\end{definition}

If $\mathcal{P}$ is some property that a group may have, such as being linear, nilpotent or solvable, then by a $\mathcal{P}$ $K$-approximate group we mean an approximate group $A$ for which the group $\langle A \rangle$ generated by $A$ has property $\mathcal{P}$.

Modulo a little fairly standard multiplicative combinatorics, the following result was proved by Hrushovski \cite[Corollary 5.10]{hrush}.

\begin{theorem}\label{hrush-thm}
Suppose that $A \subseteq \GL_n(\C)$ is a $K$-approximate group. Then $A$ is $O_{K,n}(1)$-controlled by $B$, a \emph{solvable} $K^{O(1)}$-approximate group.
\end{theorem}

By combining this with the main result of \cite{bg-solv} one immediately obtains the following extension.

\begin{theorem}\label{mainthm}
Suppose that $A \subseteq \GL_n(\C)$ is a $K$-approximate group. Then $A$ is $O_{K,n}(1)$-controlled by $B$, a \emph{nilpotent} $K^{O(1)}$-approximate group.
\end{theorem}

Hrushovski's result was proven using model theory and so did not lead to an explicit form for the $O_{K,n}(1)$ term in Theorem \ref{mainthm}, even in principle. A different proof of this result was established by the authors in \cite{bgt-long}, and this gave a stronger result in the sense that the $O_{K,n}(1)$ term was shown to vary polynomially in $K$.  We obtained a bound of the form $C_n K^{C'_n}$ where $C'_n$ could have been computed if desired (and would take the form $\exp(n^{O(1)})$), but $C_n$ could not as a consequence of our dependence on ultrafilters to prove quantitative algebraic geometry estimates in \cite{bgt-long}. A more explicit approach to these estimates is taken in the paper of Pyber and Szabo \cite{pyber-szabo} who, subsequent to our work in \cite{bgt-long}, obtain a result (Theorem 10 of their paper) which is in principle explicit.

Our aim here is to show how Theorem \ref{mainthm} follows quickly from various results in the literature, the most substantial of these being the so-called \emph{uniform Tits alternative} of the first author. The argument is largely distinct from both of the previous two proofs of Theorem \ref{mainthm}.
It gives a term $O_{K,n}(1)$ of the form $\exp(K^{O(m(n)\log K)})$, where $m(n)$ is the constant appearing in the uniform Tits alternative
(see Proposition \ref{uniform-tits} below). This is effective in principle. However, computing an explicit bound for $m(n)$ would not
only involve chasing constants in \cite{breuillard-uniformtits}, but would also require the effective version of \S 2 of \cite{breuillard-gap}
proven in \cite{breuillard-effective}. The exponential nature of $O_{K,n}(1)$ with respect to $K$ stems from the bound in Proposition \ref{sanders-cs}, which is a consequence of a recent result of Sanders \cite{sanders} and Croot-Sisask \cite{croot-sisask}.

Finally, we note that the proof of Theorem \ref{mainthm} above generalizes easily to uniformly nonamenable groups. Say that a group $G$ is $\kappa$-uniformly nonamenable\footnote{What we defined here is a slightly stronger notion of uniform non amenability than the one defined in \cite{arzh}, because we do not require $X$ to generate the group.} if whenever $X$ is a finite subset of $G$ which generates a non-amenable subgroup, then $|AX| \geq (1+\kappa)|A|$ for every finite subset $A$ in $G$ and some $\kappa>0$. The uniform Tits alternative implies that $GL_d(\C)$ is $\kappa_d$-uniformly non-amenable, for some $\kappa_d>0$. Similarly it can be shown\footnote{Strictly speaking, Koubi's theorem requires $X$ to be a generating set of $G$, however his proof extends easily to the case when $X$ generates a non-elementary subgroup.} (see Koubi \cite{koubi}) that $\delta$-hyperbolic groups are uniformly nonamenable, and it is known that their amenable subgroups are infinite cyclic-by-bounded. In this setting we obtain the following result.

\begin{theorem}\label{mainthm2}
Let $G$ be a $\kappa$-uniformly nonamenable group. Suppose that $A \subseteq G$ is a $K$-approximate group. Then $A$ is $\exp(K^{O(\log K/\kappa)})$-controlled by $B$, an \emph{amenable} $K^{O(1)}$-approximate group.
\end{theorem}

Although Theorem \ref{mainthm2} applies to $\delta$-hyperbolic groups, and thus says that any approximate subgroup is controlled by an approximate subgroup of a cyclic subgroup, it is likely that our bound can be improved a lot for these groups in the spirit of Safin's bound  \cite{safin} in the free group case (he obtained a bound of the form $O_{\eps}(K^{2+\eps})$).



\section{Six lemmas}

In this section we assemble the tools from the literature that we require to prove Theorem \ref{mainthm}. The title of the section is hardly accurate, since some of these results are quite substantial.

\begin{proposition}\label{uniform-tits}
Let $n \geq 1$ be an integer. Then there is an integer $m = m(n)$ with the following property: if $A \subseteq \GL_n(\C)$ is a finite symmetric set then either the group $\langle A \rangle$ generated by $A$ is virtually solvable, or else $A^{m}$ contains two elements generating a free subgroup of $\GL_n(\C)$.
\end{proposition}
\begin{proof} This is the ``Uniform Tits Alternative'' of the first author \cite{breuillard-uniformtits}.\end{proof}

The following result is a straightforward consequence of a lemma of Sanders \cite{sanders}. For very closely related results, see \cite{croot-sisask} and \cite{schoen}.

\begin{proposition}\label{sanders-cs}
Suppose that $A$ is a $K$-approximate group. Then for any $\eps > 0$ and $m \in \N$ there are sets $A' \subseteq A^5$, $B \subseteq A^4$ with $|A'| \geq |A|$, $|B| \gg_{K,m,\eps}|A|$ and $|A'B^{m}| \leq (1 + \eps)|A'|$.
\end{proposition}
\begin{proof}
Let $k = k(\eps,K)$ be an integer to be specified later. It is shown in Sanders \cite{sanders} that there is a symmetric set $B$ containing the identity such that $|B| \geq \exp(-K^{O(km)}) |A|$ and $B^{km} \subseteq A^4$.  We have the nesting
\[ A \subseteq A B^m \subseteq A B^{2m} \subseteq \dots \subseteq A B^{km} \subseteq A^5.\]
Note that $|A^5| \leq K^4 |A|$. Supposing that
\[ (1 + \eps)^k \geq K^4,\] it follows from the pigeonhole principle that there is some $j$, $0 \leq j  < k$ such that, setting $A' := A B^{jm}$, we have
\[ |A' B^m| \leq (1 + \eps)|A'|.\]
These sets $A'$ and $B$ then satisfy the requirements of the proposition.
\end{proof}

\begin{proposition}\label{f2-na}
Suppose that $A$ is a subset of some group $G$ and that $X$ is a further subset of $G$ containing the identity and two elements generating a nonabelian free group. Then $|AX| \geq \frac{5}{4}|A|$.
\end{proposition}
\begin{proof}
This is basically the observation, originally due to von Neumann, made whilst proving that every group containing a non-abelian free subgroup $F_2$ is not amenable using F{\o}lner's criterion; see for example the notes of the third author \cite{tao-amenable-blog} or the remarks on page 4 of \cite{breuillard-uniformtits}.
\end{proof}

\begin{proposition}\label{malcev-platanov}
Suppose that $G \leq \GL_n(\C)$ is virtually solvable. Then there is a solvable subgroup $H \leq G$ such that $[G : H] = O_n(1)$.
\end{proposition}
\begin{proof} This result of Mal'cev-Platonov is proved in \cite[Appendix B]{bgt-long}.\end{proof}

\begin{proposition}\label{solv-to-nil}
Let $K \geq 2$, and suppose that $A$ is a $K$-approximate \emph{solvable} subgroup of $\GL_n(\C)$. Then $A$ is $K^{O_n(1)}$-controlled by a $K^{O(1)}$-approximate \emph{nilpotent} subgroup.
\end{proposition}
\begin{proof} This is essentially the main result of \cite{bg-solv}. There is one difference, in that here we obtain a $K^{O(1)}$-approximate group rather than a $K^{O_n(1)}$-approximate group. However, this stronger statement follows easily from the weaker one and the additive-combinatorial lemma below. \end{proof}

Finally, we require a standard additive combinatorics lemma.

\begin{lemma}\label{standard-add-comb}
Let $0 < \delta < 1$, $k\geq 1$ and $K \geq 2$ be parameters. Suppose that $A$ is a $K$-approximate group of some ambient group $G$, and that $H \leq G$ is a subgroup with the property that $A^{k}$ intersects some coset $Hx$ in a set of size $\delta |A|$. Then $A^2 \cap H$ is a $2K^3$-approximate group which $2K^{2k+4}/\delta$-controls $A$.
\end{lemma}
\begin{proof}
According to \cite[Lemma 3.3]{bg-compact}, $B:=A^2 \cap H$ is a $2K^3$-approximate group and $|A^{2k} \cap H| \leq K^{2k-1}|B|$. Moreover $|AB|\leq |A^3|\leq K^2|A|\leq \frac{K^{2k+1}}{\delta}|B|$. An appeal to Ruzsa's covering lemma ends the proof.
\end{proof}

\section{Proof of the main theorem}

This is a very short task given the ingredients we assembled in the previous section. Let $A \subseteq \GL_n(\C)$ be a $K$-approximate group. By Proposition \ref{sanders-cs}, applied with $\eps = 0.1$ (say) and with $m = m(n)$ the quantity appearing in Proposition \ref{uniform-tits}, there are sets $A' \subseteq A^5$, $B \subseteq A^4$ with $|A'| \geq |A|$ and $|B| \gg_{n,K}|A|$ such that $|A' B^m| < \frac{5}{4}|A'|$. By Proposition \ref{f2-na}, $B^m$ cannot contain two elements generating a nonabelian free group. By the uniform Tits alternative, $B$ must generate a virtually solvable group $G$. By Proposition \ref{malcev-platanov}, this $G$ contains a solvable subgroup $H$ with $[G : H] = O_n(1)$, and so by the pigeonhole principle there is some $x$ such that $|B \cap Hx| \gg_n |B|$. This implies that
\[ |A^4 \cap Hx| \geq |B \cap Hx| \gg_n |B| \gg_{n,K} |A|.\]
At this point we may apply Lemma \ref{standard-add-comb} to conclude that $A^2 \cap H$ is a $2K^3$-approximate solvable group which $O_{K,n}(1)$-controls $A$. Finally, an application of Proposition \ref{solv-to-nil} completes the proof of Theorem \ref{mainthm}.\endproof\vspace{11pt}

We leave the proof of Theorem \ref{mainthm2}, which proceeds along almost identical lines but does not require Proposition \ref{malcev-platanov} or Proposition \ref{solv-to-nil}, to the reader.\vspace{11pt}

\emph{Acknowledgement.} We thank Thomas Delzant for useful discussions regarding hyperbolic groups.

\setcounter{tocdepth}{1}

\end{document}